\documentclass[a4paper,11pt]{amsart}

\usepackage{amsmath,amsthm,amsfonts,amssymb,color,enumerate,tikz,bbm,hyperref}
\usepackage[margin=1in]{geometry}
\usepackage[capitalize]{cleveref}
\usepackage[british]{babel}
\usetikzlibrary{decorations.pathmorphing}

\theoremstyle{plain}
\newtheorem{thm}{Theorem}[section]
\newtheorem{lem}[thm]{Lemma}
\newtheorem{prop}[thm]{Proposition}

\theoremstyle{definition}
\newtheorem{defn}[thm]{Definition}
\newtheorem*{ack}{Acknowledgements}
\theoremstyle{remark}
\newtheorem{rmk}[thm]{Remark}
\numberwithin{equation}{section}
\crefname{lem}{Lemma}{Lemmas}
\crefname{thm}{Theorem}{Theorems}

\newcommand{\N}{\mathbb{N}}
\newcommand{\Z}{\mathbb{Z}}
\newcommand{\Q}{\mathbb{Q}}
\newcommand{\C}{\mathbb{C}}
\newcommand{\R}{\mathbb{R}}
\newcommand{\bfa}{\mathbf{a}}
\newcommand{\bfe}{\mathbf{e}}
\newcommand{\bfu}{\mathbf{u}}
\newcommand{\bfv}{\mathbf{v}}
\newcommand{\bfeps}{\boldsymbol\varepsilon}
\newcommand{\bfzero}{\mathbf{0}}
\newcommand{\dc}{\operatorname{dc}}
\newcommand{\Sym}{\mathrm{Sym}}

\title[Negligibility of elliptic elements]{Negligibity of elliptic elements in ascending HNN-extensions of $\Z^m$}
\author{Motiejus Valiunas}
\address{Mathematical Sciences, University of Southampton, University Road, Southampton SO17 1BJ, United Kingdom}
\email{m.valiunas@soton.ac.uk}
\date{\today}

\begin{document}

\begin{abstract}
We study ascending HNN-extensions $G$ of finitely generated free abelian groups: examples of such $G$ include soluble Baumslag-Solitar groups and fundamental groups of orientable prime $3$-manifolds modelled on Sol geometry. In particular, we study the elliptic subgroup $A \leq G$, consisting of all elements that stabilise a point in the Bass-Serre tree of $G$. We consider the density of $A$ with respect to ball counting measures corresponding to finite generating sets of $G$, and we show that $A$ is exponentially negligible in $G$ with respect to such sequences of measures. As a consequence, we show that the set of tuples $(x_0,\ldots,x_r) \in G^{r+1}$, such that the $(r+1)$-fold simple commutator $[x_0,\ldots,x_r]$ vanishes, is exponentially negligible in $G^{r+1}$ with respect to sequences of ball counting measures.
\end{abstract}

\maketitle
\tableofcontents

\section{Introduction}

The growth of soluble Baumslag-Solitar groups, $G = BS(1,N) = \langle a,t \mid tat^{-1} = a^N \rangle$, has been widely studied. For instance, it has been shown that they have rational growth with respect to the standard generating set, with an explicitly calculated growth series \cite{ceg}. Moreover, the growth series for a `higher dimensional' generalisation of $BS(1,3)$ -- the ascending HNN-extension of $\Z^m$ given by the `cubing homomorphism' $\bfu \mapsto 3\bfu$ -- has been calculated in \cite{sanchez} with respect to standard generators.

Some results on growth of Baumslag-Solitar groups that are independent of the choice of a generating set are also known; in particular, the minimal exponential growth rates (with respect to an arbitrary finite generating set) for $BS(1,N)$ have been calculated in \cite{bucher}. This paper aims to provide additional results of this nature.

Growth of horocyclic subgroups $\langle a \rangle \cong \Z$ of Baumslag-Solitar groups $G = BS(p,q)$ has been also studied, as this is usually the first step to understanding the growth of $G$ itself: see \cite{fks}. Here we study the growth of the normal closure $A = \langle\!\langle a \rangle\!\rangle \cong \Z\left[\frac1N\right]$ of the horocyclic subgroup of $BS(1,N)$. We say a subset $\mathcal{A} \subseteq G$ is \emph{exponentially negligible} in a group $G$ with respect to a finite generating set $X$ if the proportion of elements that are inside $\mathcal{A}$, counted over the ball of radius $n$ in the Cayley graph $\Gamma(G,X)$, tends to zero exponentially fast as $n \to \infty$ (see \cref{d:negl}).

\begin{thm}[see {\cref{t:A}}] \label{t:A-BS}
Let $G = BS(1,N)$ with $|N| \geq 2$ and let $A \lhd G$ be the normal closure of the horocyclic subgroup in $G$. Then $A$ is exponentially negligible in $G$ with respect to any finite generating set.
\end{thm}

Another class of groups this paper aims to study are certain cyclic extensions of $\Z^m$. In particular, we study the semidirect products $G = \Z^m \rtimes \Z = \Z^m \rtimes_T \Z$, where the action of the generator $1 \in \Z$ on $\Z^m$ is represented by a matrix $T \in SL(m,\Z)$ that has an eigenvalue $\lambda \in \C$ with $|\lambda| > 1$. In the case $m = 2$, such groups are known to be $3$-manifold groups: they are fundamental groups of orientable prime $3$-manifolds modelled on Sol geometry -- one of the eight geometries of $3$-manifolds appearing in Thurston's Geometrisation Conjecture \cite{scott}.

Such groups $G = \Z^m \rtimes_T \Z$ are soluble (in fact, metabelian), but not virtually nilpotent; in particular, such a group $G$ must have exponential growth.
There has been some interest in growth properties of these groups in the case $m = 2$. Specifically, it is known that a finite-index subgroup of such a group $G$ has rational growth with respect to some generating set \cite{putman}, and in certain cases (in which $G$ itself has rational growth) the growth series of $G$ was explicitly computed \cite{parry}. Here we study growth of the base subgroup $\Z^m$ of the extension $G = \Z^m \rtimes_T \Z$.

\begin{thm}[see {\cref{t:A}}] \label{t:A-3m}
Let $G = \Z^m \rtimes_T \Z$ be a cyclic extension of $\Z^m$, where $T \in SL(m,\Z)$ has an eigenvalue $\lambda \in \C$ with $|\lambda| > 1$. Then the base group $\Z^m$ is exponentially negligible in $G$ with respect to any finite generating set.
\end{thm}

As an application of \cref{t:A-BS,t:A-3m}, we study `probabilistic nilpotence' of these classes of groups. In particular, for $r \in \N$ and a group $G$, define
\begin{equation} \label{e:NrG}
\mathcal{N}_r(G) = \{ (x_0,\ldots,x_r) \in G^{r+1} \mid [x_0,\ldots,x_r] = 1 \} \subseteq G^{r+1},
\end{equation}
where we define $[x_0,x_1] = x_0^{-1} x_1^{-1} x_0 x_1$, and, inductively,
\[ [x_0,\ldots,x_r] = [[x_0,\ldots,x_{r-1}],x_r]. \]

In \cite[Theorem 1.9]{mtvv}, Martino, Tointon, Ventura and the author showed that if a group $G$ is finitely generated and not virtually nilpotent, then the probability that a random walk on a Cayley graph of $G^{r+1}$ will end in $\mathcal{N}_r(G)$ after $n$ steps tends to zero as $n \to \infty$. It is not known whether the same result holds if a `random walk measure' is replaced by a `ball counting measure' -- see \cite[Question 1.32]{mtvv} and discussion before \cref{t:dn}. However, the following result answers this question affirmatively for our particular classes of groups.

\begin{thm}[see {\cref{t:dn}}] \label{t:dn-spec}
Let $G$ be either $BS(1,N)$ with $|N| \geq 2$ or the cyclic extension $\Z^m \rtimes_T \Z$ where $T \in SL(m,\Z)$ has an eigenvalue $\lambda \in \C$ with $|\lambda| > 1$. Then $\mathcal{N}_r(G)$ is exponentially negligible in $G^{r+1}$ with respect to any finite generating set of $G$.
\end{thm}

In \cref{s:mres}, we give statements of the main results of this paper, \cref{t:A,t:dn}, which apply for ascending HNN-extensions of $\Z^m$. An observation that the latter class of groups includes both the Baumslag-Solitar groups $BS(1,N)$ and the cyclic extensions $\Z^m \rtimes \Z$ mentioned above allows us to deduce \cref{t:A-BS,,t:A-3m,t:dn-spec}. In \cref{s:Ggrowth}, we classify ascending HNN-extensions of $\Z^m$ in terms of their growth: see \cref{p:Ggrowth}. We prove \cref{t:A} in \cref{s:A} and \cref{t:dn} in \cref{s:dn}.

\begin{ack}
The author would like to thank his PhD supervisor, Armando Martino, without whose guidance and support this work would not have been completed, and the anonymous referee for valuable comments.
\end{ack}

\section{Main definitions and results} \label{s:mres}

To study soluble Baumslag-Solitar groups and cyclic extensions $\Z^m \rtimes \Z$, we consider a more general class of groups. In particular, we study \emph{ascending HNN-extensions of finitely generated free abelian groups}. Such groups can be parametrised by square matrices with integer entries and non-zero determinant:

\begin{defn} \label{d:GmT}
Let $m \in \N$ and let $T$ be an $m \times m$ matrix with integer entries and $\det T \neq 0$.
\begin{enumerate}
\item Define the group $G(m,T)$ by the presentation
\[ G(m,T) = \left\langle a_1, \cdots, a_m, t \:\middle|\: \begin{array}{@{}c@{}} a_ia_j = a_ja_i \text{ for } 1 \leq i < j \leq m \\ t \bfa^\bfu t^{-1} = \bfa^{T\bfu} \text{ for all } \bfu \in \Z^m \end{array} \right\rangle, \]
where $\bfa^\bfv$ denotes $a_1^{v_1} \cdots a_m^{v_m}$ for any $\bfv = (v_1,\ldots,v_m) \in \Z^m$. It is easy to see that each element of $G$ can be expressed (although non-uniquely) as $t^{-r}\bfa^\bfu t^s$ for some $r,s \geq 0$ and $\bfu \in \Z^m$.

\item Consider the homomorphism $\tau: G(m,T) \to \Z$ given by the $t$-exponent sum: let $\tau(a_i) = 0$ for $1 \leq i \leq m$ and $\tau(t) = 1$. Define the \emph{elliptic subgroup} $A(m,T)$ of $G(m,T)$ as
\[ A(m,T) = \ker\tau \trianglelefteq G(m,T); \]
notice that $A(m,T)$ is abelian.
\end{enumerate}
\end{defn}

\begin{rmk} \label{r:hnn}
As an alternative construction, note that $G(m,T)$ can be expressed as an ascending HNN-extension of $\Z^m$:
\[ G(m,T) \cong (\Z^m) \ast_\phi, \]
where $\phi: \Z^m \xrightarrow\cong T(\Z^m)$ is defined by $\phi(\bfu) = T\bfu$. We may define $A(m,T)$ in terms of the action of the group $G(m,T)$ on its Bass-Serre tree $\mathcal{T}$, corresponding to this HNN-decomposition. In particular, we define $A(m,T)$ to be the union of all point stabilisers with respect to the action of $G(m,T)$ on $\mathcal{T}$. Since this action fixes an end of $\mathcal{T}$, the set $A(m,T)$ turns out to be a (normal) subgroup.
\end{rmk}

Notice that the family of groups $G(m,T)$ includes both soluble Baumslag-Solitar groups and aforementioned cyclic extensions. Indeed, for $m = 1$ and $T = \begin{pmatrix} N \end{pmatrix}$ for some non-zero $N \in \Z$ we have $G(m,T) \cong BS(1,N)$, whereas if $T \in SL(n,\Z)$ then $G(m,T) \cong \Z^m \rtimes_T \Z$. Moreover, the `higher Baumslag-Solitar groups' studied in \cite{sanchez} are just $G(m,3I_m)$, where $I_m$ is the identity matrix.

It is worth mentioning that in the case $T \in SL(2,\Z)$, if in addition $m = 2$ then the groups $G(m,T)$ appear in the theory of $3$-manifolds. Specifically, we have $G(2,T) \cong \pi_1(M)$ for an orientable prime $3$-manifold $M$ modelled on Euclidean, Nil or Sol geometry -- these are three out of eight geometries appearing in the classification of $3$-manifolds; see \cite{scott}. Moreover, the geometry on which $M$ is modelled can be recognised from $T$: it is Sol if $T$ has two distinct real eigenvalues, Nil if $T$ is not diagonalisable, and Euclidean otherwise.

In the present paper we are interested in the growth of the elliptic subgroup $A(m,T)$ relative to the growth of $G(m,T)$. As mentioned above, this subgroup is the normal closure of the horocyclic subgroup in a soluble Baumslag-Solitar group in the case $m = 1$, or the base subgroup of the cyclic extension $\Z^m \rtimes_T \Z$ if $T \in SL(m,\Z)$. In order to study the growth of $A(m,T)$, we first need to introduce some terminology.

Fix a finitely generated infinite group $G$ and let $Y$ be a finite generating set of $G$. This allows us to define the \emph{word metric} $|\cdot|_Y$ for $G$, by letting $|g|_Y$ to be the minimal word-length of $g \in G$ with respect to $Y$. For $n \in \Z_{\geq 0}$, let
\[ B_Y(n) = \{ g \in G \mid |g|_Y \leq n \} \]
be the \emph{ball} in $G$ with respect to $Y$ of radius $n$.

For any $r \in \N$, we may characterise `small' and `large' subsets of the $r$-fold direct product $G^r = \overbrace{G \times \cdots \times G}^r$ by using ball counting measures, as follows.

\begin{defn} \label{d:negl}
Let $Y \subset G$ be a finite generating set, let $r \in \N$, and let $\mathcal{A} \subseteq G^r$ be a subset. Let
\[ \gamma_n^Y(\mathcal{A}) = \frac{|\mathcal{A} \cap B_Y(n)^r|}{|B_Y(n)|^r}. \]
\begin{enumerate}
\item We say $\mathcal{A}$ is \emph{negligible} in $G^r$ with respect to $Y$ if
\[ \limsup_{n \to \infty} \gamma_n^Y(\mathcal{A}) = 0. \]
\item We say $\mathcal{A}$ is \emph{exponentially negligible} in $G^r$ with respect to $Y$ if
\[ \limsup_{n \to \infty} \sqrt[n]{\gamma_n^Y(\mathcal{A})} < 1. \]
It is clear that an exponentially negligible subset of $G^r$ is also negligible.
\item We say $\mathcal{A}$ \emph{has exponential growth} in $G^r$ with respect to $Y$ if
\[ \liminf_{n \to \infty} \sqrt[n]{|\mathcal{A} \cap B_Y(n)^r|} > 1. \]
\end{enumerate}
\end{defn}

These concepts are not new: for instance, closely related notions of natural density and exponential density of subsets in groups were introduced in \cite{burillo}. Similar, although not equivalent, notions of negligible and strongly negligible subsets of direct products of groups were considered in \cite{kapovich}.

\begin{rmk}
It is easy to construct subsets $\mathcal{A} \subseteq G^r$ of a group $G$ of exponential growth such that $\mathcal{A}$ is (exponentially) negligible with respect to some generating set but not with respect to another one. For instance, if $G = F_2 \times F_3$ is a direct product of two free groups of ranks 2 and 3, and if $X_i$ is a basis for $F_i$ ($i \in \{2,3\}$), then $\{1\} \times F_3$ is not negligible in $G$ with respect to the `union' of the generating sets, $Y_\cup = (X_2 \times \{1\}) \cup (\{1\} \times X_3)$, but exponentially negligible with respect to their `product', $Y_\times = (X_2^{\pm 1} \cup \{1\}) \times (X_3^{\pm 1} \cup \{1\})$. Indeed, note that we have
\[
B_{Y_\cup}(n) = \bigcup_{i=0}^n \left( B_{X_2}(i) \times B_{X_3}(n-i) \right) \qquad \text{and} \qquad B_{Y_\times}(n) = B_{X_2}(n) \times B_{X_3}(n).
\]
This allows us to calculate sizes of balls in $G$ and their intersections with $\{1\} \times F_3$ explicitly to obtain asymptotics
\[
|B_{Y_\cup}(n)|, |(\{1\} \times F_3) \cap B_{Y_\cup}(n)|, |(\{1\} \times F_3) \cap B_{Y_\times}(n)| \sim 5^n \qquad \text{and} \qquad |B_{Y_\times}(n)| \sim 15^n,
\]
where we write $f(n) \sim g(n)$ if $\frac{f(n)}{g(n)} \to C$ for some $C \in (0,\infty)$ as $n \to \infty$.

On the other hand, note that the inequality $\liminf_{n \to \infty} \sqrt[n]{|\mathcal{A} \cap B_Y(n)^r|} > 1$ is independent of the generating set $Y$ (as all word metrics on $G$ are bi-Lipschitz equivalent), and hence if a subset has exponential growth with respect to some finite generating set, then it has exponential growth with respect to all of them. Thus we may simply say that $\mathcal{A}$ \emph{has exponential growth} in $G^r$ (without referring to a particular generating set).
\end{rmk}

We now return to the case of a group $G = G(m,T)$ and its elliptic subgroup $A = A(m,T)$, as in \cref{d:GmT}. It is well-known that in many cases, $A$ will have exponential growth in $G$: that is, we have $\liminf_{n \to \infty} \sqrt[n]{|A \cap B_Y(n)|} > 1$ for every finite generating set $Y$ of $G$. More precisely, we have

\begin{prop} \label{p:Ggrowth}
\begin{enumerate}
\item If all eigenvalues of $T$ are equal to $1$, then the group $G(m,T)$ is nilpotent.
\item If absolute values of all eigenvalues of $T$ are equal to $1$, then the group $G(m,T)$ is virtually nilpotent.
\item Otherwise, $G(m,T)$ has exponential growth, and $A(m,T)$ has exponential growth in $G(m,T)$.
\end{enumerate}
\end{prop}

We prove \cref{p:Ggrowth} in \cref{s:Ggrowth}.

As a corollary of \cref{p:Ggrowth}, we immediately obtain the well-known facts that the group $BS(1,N)$ with $|N| \geq 2$ and the fundamental group of a $3$-dimensional Sol-manifold both have exponential growth. These facts allow us to deduce \cref{t:A-3m,t:A-BS} from \cref{t:A}, as well as \cref{t:dn-spec} from \cref{t:dn}.

\begin{rmk}
We do not claim all of the statements in \cref{p:Ggrowth} to be original: for instance, if $|\det(T)| \geq 2$ then the Bass-Serre tree of $G$ (see \cref{r:hnn}) has infinitely many ends, and in this case it is known (see \cite{hb}) that $G$ has uniformly exponential growth (and hence exponential growth). Nevertheless, we are not aware of any reference in the literature which includes these statements, and thus we prove \cref{p:Ggrowth} here for completeness.
\end{rmk}

Note that the group $G/A \cong \Z$  has linear growth -- in particular, the number of cosets of $A$ in $G$ that intersect $B_Y(n)$ non-trivially grows linearly with $n$ (for any finite generating set $Y$ of $G$). Hence we might expect $A$ to be `large' in $G$. However, our first general main result states:

\begin{thm} \label{t:A}
If $G(m,T)$ has exponential growth, then $A(m,T)$ is exponentially negligible in $G(m,T)$ with respect to any finite generating set.
\end{thm}

We prove \cref{t:A} in \cref{s:A}.

\begin{rmk}
Note that $G = G(m,T)$ is metabelian, as $A = A(m,T)$ is abelian and the sequence
\begin{equation} \label{e:metab}
1 \to A \hookrightarrow G \xrightarrow\tau \Z \to 1
\end{equation}
is exact; furthermore, the vector space $A \otimes \Q$ is $m$-dimensional over $\Q$. We suspect that, under some additional technical results, our proof of \cref{t:A} generalises to any groups $G$ and $A \unlhd G$ such that $A$ is abelian, \eqref{e:metab} is exact, and $A \otimes \Q$ is finite-dimensional. In other words, we may expect to generalise the argument in the case when the matrix $T$ has rational entries that are not necessarily integers. Details of this are left to the interested reader.
\end{rmk}

As an application of \cref{t:A}, we study `probabilistic nilpotence' of the groups $G(m,T)$. In particular, for $r \in \N$ and a group $G$, define $\mathcal{N}_r(G) \subseteq G^{r+1}$ as in \eqref{e:NrG}. Notice that $\mathcal{N}_r(G) = G^{r+1}$ if and only if $G$ is nilpotent of class at most $r$. This suggests that measuring $\mathcal{N}_r(G)$ will tell us how close to being nilpotent a group is. In \cite{mtvv}, Martino, Tointon, Ventura and the author defined the \emph{degree of $r$-nilpotence} of $G$ with respect to a sequence of measures $M = (\mu_n)_{n=1}^\infty$ on $G$ as
\[ \dc^k_M(G) = \limsup_{n \to \infty} (\mu_n \times \cdots \times \mu_n)(\mathcal{N}_r(G)). \]

It was shown in \cite[Theorem 1.8]{mtvv} that if $G$ is finitely generated and the sequence $(\mu_n)$ measures index uniformly -- that is, $\mu_n(xH) \to [G:H]^{-1}$ uniformly over all $x \in G$ and all subgroups $H \leq G$, where by convention $[G:H]^{-1} = 0$ if $H$ has infinite index in $G$ -- then $\dc^k_M(G) > 0$ if and only if $G$ is virtually $k$-step nilpotent. A particular case of special interest of this are the `random walk measures': it was shown in \cite[Theorem 14]{tointon} that if $\mu$ is a symmetric, finitely supported generating probability measure on $G$ with $\mu(1) > 0$, and if $M = (\mu^{*n})$ is the sequence of measures corresponding to the steps of the random walk on $G$ with respect to $\mu$, then $M$ measures index uniformly on $G$.

Instead of random walk measures, here we consider `ball counting measures': that is, we replace the sequence $(\mu^{*n})$ with the sequence $(\gamma_n^Y)$ as in \cref{d:negl}. Intuitively, this can be thought of as avoiding `overcounting' whilst measuring a subset: while $\mu^{*n}(\mathcal{A})$ counts \emph{all} random walks of length $n$ ending up in $\mathcal{A} \cap B_Y(n)$, where $Y$ is the support of $\mu$, the number $\gamma_n^Y(\mathcal{A})$ counts every element of $\mathcal{A} \cap B_Y(n)$ exactly once.

In general, the sequence $(\gamma_n^Y)_{n=1}^\infty$ need not measure index uniformly, even if we restrict to groups of the form $G(m,T)$ -- see \cref{r:notunif}. However, we do not know if there exists a finitely generated group $G = \langle Y \rangle$ that is not virtually nilpotent, such that $\mathcal{N}_r(G)$ is not negligible in $G^{r+1}$ with respect to $Y$ -- see \cite[Question 1.32]{mtvv}. The following result follows from \cref{t:A} and shows that a group of the form $G(m,T)$ of exponential growth cannot be taken as such an example, even though (by \cref{p:Ggrowth}) it has an abelian subgroup $A(m,T)$ of exponential growth in $G(m,T)$.

\begin{thm} \label{t:dn}
If $G = G(m,T)$ has exponential growth, then $\mathcal{N}_r(G)$ is exponentially negligible in $G^{r+1}$ with respect to any finite generating set.
\end{thm}

We prove \cref{t:dn} in \cref{s:dn}.

\begin{rmk} \label{r:notunif}
To show that \cref{t:dn} does not follow directly from \cite[Theorem 1.8]{mtvv}, consider the following example. Let $T = \begin{pmatrix} 2 & 0 \\ 0 & 1 \end{pmatrix}$, and let
\[
G = G\left( 2,T \right) = \langle a_1,a_2,t \mid [a_1,a_2] = [t,a_2] = 1, ta_1t^{-1} = a_1^2 \rangle \cong BS(1,2) \times \Z.
\]
Let $X = \{ a_1,a_2,t \}$ be the standard generating set. Let $H = \langle a_1,t \rangle \leq G$, so that $G = H \times \langle a_2 \rangle$ and $H \cong BS(1,2)$. It is known that if $B_{\{a_1,t\}}(n)$ is a ball in $H$ with respect to $\{a_1,t\}$ of radius $n$, then $B_{\{a_1,t\}}(n)/2^n \to C$ for some $C \in (0,\infty)$: see \cite{ceg}. A calculation then shows that $\limsup_{n \to \infty} \gamma_n^X(H) = \frac{1}{3} > 0$, even though $H$ has infinite index in $G$. In particular, the sequence of measures $(\gamma_n^X)$ does not measure index uniformly, and so we cannot apply \cite[Theorem 1.8]{mtvv} in this case.
\end{rmk}

\section{Classification of the groups \texorpdfstring{$G(m,T)$}{G(m,T)}} \label{s:Ggrowth}

In this section we specify which groups \cref{t:A,t:dn} can be applied to. Specifically, we notice that any group $G(m,T)$ is either virtually nilpotent or has exponental growth; by Gromov's Polynomial Growth Theorem \cite{gromov81}, this is just saying that there are no groups $G(m,T)$ of intermediate growth. More precisely, we give a necessary and sufficient condition on the matrix $T$ for the group $G(m,T)$ to have exponential growth (see \cref{p:Ggrowth}).

For the rest of the paper, fix an integer $m \in \N$ and an $m \times m$ matrix $T$, and write $G = G(m,T)$ and $A = A(m,T)$. Let $\lambda_1,\ldots,\lambda_m \in \C$ be the eigenvalues of $T$ (counted with multiplicity), and suppose without loss of generality that $|\lambda_1| \leq \cdots \leq |\lambda_m|$. Let $\lambda = |\lambda_m|$, and note that since
\[ \lambda^m \geq \prod_{i=1}^m |\lambda_i| = |\det T| \geq 1, \]
we have $\lambda \geq 1$, with equality if and only if $|\lambda_i| = 1$ for all $i$.

The following two lemmas will be used in the proofs of \cref{p:Ggrowth,t:A}. The first of these is easy to check and its proof is left as an exercise.

\begin{lem} \label{l:phi}
The map $\varphi$ defined by setting $\varphi(t^n\bfa^\bfu t^{-n}) = T^n\bfu$ for $n \in \Z$ and $\bfu \in \Z^m$ can be extended to an injective homomorphism $\varphi: A \to \Q^m$. \qed
\end{lem}

The next Lemma allows us to construct exponentially many elements in $G$ of given word-length.

\begin{lem} \label{l:geps}
If $|\lambda_i| \neq 1$ for some $i$, then there exist constants $R \in \N$ and $j \in \{ 1,\ldots,m \}$ with the following property. For any $k \in \N$ and $\bfeps = (\varepsilon_0,\ldots,\varepsilon_k) \in \{0,1\}^{k+1}$, let
\[ g_{\bfeps} = a_j^{\varepsilon_0} t^R a_j^{\varepsilon_1} t^R \cdots t^R a_j^{\varepsilon_k} \in G. \]
Then $g_{\bfeps} \neq g_{\hat\bfeps}$ for any two distinct elements $\bfeps,\hat\bfeps \in \{0,1\}^{k+1}$.
\end{lem}

\begin{proof}
Let $\varphi: A \to \Q^m$ be as in \cref{l:phi}. For $g_{\bfeps},g_{\hat\bfeps}$ as in the statement, we have
\[ \varphi(g_{\bfeps}g_{\hat\bfeps}^{-1}) = (\varepsilon_0-\hat\varepsilon_0) \bfe_j + (\varepsilon_1-\hat\varepsilon_1) T^R\bfe_j + \cdots + (\varepsilon_k-\hat\varepsilon_k) T^{kR}\bfe_j, \]
where $\{ \bfe_i \mid 1 \leq i \leq n \}$ is the standard basis for $\Z^m$, so that $a_i = \bfa^{\bfe_i}$. Thus, by injectivity of $\varphi$, it is enough to find $R \in \N$ and $j \in \{ 1,\ldots,m \}$ such that
\[
\| T^{kR} \bfe_j \| > \sum_{i=0}^{k-1} \| T^{iR} \bfe_j \|
\]
for all $k \in \N$ with a suitable choice of norm $\|\cdot\|$. In particular, it is enough to require
\begin{equation} \label{e:Tsk}
\frac{\| T^{kR} \bfe_j \|}{\| T^{(k-1)R} \bfe_j \|} \geq 2
\end{equation}
for all $k \in \N$. We will use Jordan normal forms to define a norm $\| \cdot \|$ and to approximate $\| T^{kR} \bfe_j \|$ for $k$ large.

Let $x \in \{ 1,\ldots,m-1 \}$ be such that $|\lambda_x| < |\lambda_{x+1}| = |\lambda_{x+2}| = \cdots = |\lambda_m|$ (and $x = 0$ if $|\lambda_1| = |\lambda_m|$). Let $S \in \N$ be such that both $|\lambda_x|^S + 1 \leq \frac{1}{2} |\lambda_m|^S$ and $|\lambda_m|^S \geq 2$, and consider the Jordan normal form $PT^SP^{-1}$ for $T^S$ (where $P \in GL_m(\C)$): it is a block-diagonal matrix with blocks $X,Y_1,\ldots,Y_z$ where
\begin{equation*}
X = \begin{pmatrix}
\lambda_1^S & \alpha_1 \\
& \lambda_2^S & \ddots \\
&& \ddots & \alpha_{x-1} \\
&&& \lambda_x^S
\end{pmatrix} \in GL_x(\C)
\qquad\text{and}\qquad
Y_i = \begin{pmatrix}
\hat\lambda_i^S & 1 \\
& \hat\lambda_i^S & \ddots \\
&& \ddots & 1 \\
&&& \hat\lambda_i^S
\end{pmatrix} \in GL_{y_i}(\C)
\end{equation*}
for some $\alpha_1,\ldots,\alpha_{x-1} \in \{ 0,1 \}$ and some $y_1,\ldots,y_z \in \N$ where, without loss of generality, $y_1 \leq \cdots \leq y_z$, and where $|\hat\lambda_i| = |\lambda_m|$. Furthermore, let $w \in \{1,\ldots,z-1\}$ be such that $y_w < y_{w+1} = \cdots = y_z$ (and $w = 0$ if $y_1 = y_z$), let $\lambda = |\lambda_m|$ and let $y = y_z$. Define a norm on $\C^m$ by
\[ \| \bfu \| = \| P\bfu \|_\infty = \max \{ |(P\bfu)_i| \mid 1 \leq i \leq m \}, \]
and let $j \in \{1,\ldots,m\}$ be such that the last entry of $P\bfe_j$ is non-zero (such a choice is possible since the $P\bfe_i$ span $\C^m$). The idea is now to approximate $\| T^{nS}\bfe_j \|$ by a constant multiple of $\genfrac(){0pt}{1}{nS}{y-1} \lambda^{nS-y+1}$ when $n$ is large.

Since the $\ell_1$-norm of any row of $X$ is at most $|\lambda_x|^S+1 \leq \frac{\lambda^S}{2}$, we get $\| X\bfu \|_\infty \leq \frac{\lambda^S}{2} \|\bfu\|_\infty$ for any $\bfu \in \C^x$, and so $\| X^n\bfu \|_\infty \leq \left(\frac{\lambda^S}{2}\right)^n \|\bfu\|_\infty$. Moreover, we have
\[
Y_i^n = \begin{pmatrix}
\hat\lambda_i^n & n \hat\lambda_i^{n-1} & \genfrac(){0pt}{1}{n}{2} \hat\lambda_i^{n-2} & \cdots & \genfrac(){0pt}{1}{n}{y_i-1} \hat\lambda_i^{n-y_i+1} \\
& \hat\lambda_i^n & n \hat\lambda_i^{n-1} & \cdots & \genfrac(){0pt}{1}{n}{y_i-2} \hat\lambda_i^{n-y_i+2} \\
&& \ddots & \ddots & \vdots \\
&&& \hat\lambda_i^n & n \hat\lambda_i^{n-1} \\
&&&& \hat\lambda_i^n
\end{pmatrix}.
\]

Now let $w \in \{1,\ldots,z-1\}$ be such that $y_w < y_{w+1} = \cdots = y_z$ (with $w = 0$ if $y_1 = y_z$), and let $y = y_z$. Then the above calculations imply that if we denote
\[ PT^{nS}\bfe_j = (\beta_1^{(n)},\ldots,\beta_x^{(n)},\gamma_{1,1}^{(n)},\ldots,\gamma_{1,y_1}^{(n)},\ldots,\gamma_{z,1}^{(n)},\ldots,\gamma_{z,y_z}^{(n)}), \]
then
\begin{equation*}
\lim_{n \to \infty} \frac{|\beta_i^{(n)}|}{\genfrac(){0pt}{1}{nS}{y-1} \lambda^{nS-y+1}} = 0
\end{equation*}
and
\begin{equation*}
\lim_{n \to \infty} \frac{|\gamma_{i,l}^{(n)}|}{\genfrac(){0pt}{1}{nS}{y-1} \lambda^{nS-y+1}} = \begin{cases} |\gamma_{i,y_i}^{(0)}| & \text{if } l = 1 \text{ and } i > w, \\ 0 & \text{otherwise}. \end{cases}
\end{equation*}
Since by the choice of $j$ we have $\gamma_{z,y_z}^{(0)} \neq 0$, it follows that there exists a constant $n_0 \in \N$ such that for $n \geq n_0$ we have
\begin{equation*}
\frac{1}{2} \leq \frac{\| T^{nS}\bfe_j \|}{\genfrac(){0pt}{1}{nS}{y-1} \lambda^{nS-y+1} \max \{ |\gamma_{i,y_i}^{(0)}| \mid w+1 \leq i \leq z \}} \leq \frac{3}{2}.
\end{equation*}
By increasing $n_0$ further, we may also assume that $n_0 \geq 3$ and $\| T^{n_0S}\bfe_j \| \geq 2 \| \bfe_j \|$. Let $R = n_0S$. Then, since $|\lambda_x|^S+1 \leq \frac{\lambda^S}{2}$ and in particular $\lambda^S \geq 2$, we have
\[
\frac{\| T^{kR}\bfe_j \|}{\| T^{(k-1)R}\bfe_j \|} \geq \frac{\frac{1}{2}\genfrac(){0pt}{1}{kn_0S}{y-1} \lambda^{kn_0S-y+1}}{\frac{3}{2}\genfrac(){0pt}{1}{(k-1)n_0S}{y-1} \lambda^{(k-1)n_0S-y+1}} \geq \frac{\lambda^{n_0S}}{3} \geq \frac{\lambda^{3S}}{3} \geq \frac{8}{3} > 2
\]
for all $k \geq 2$, and also, by assumption on $n_0$, $\| T^{R}\bfe_j \| / \| \bfe_j \| \geq 2$. Thus \eqref{e:Tsk} holds, as required.
\end{proof}

\begin{proof}[Proof of \cref{p:Ggrowth}]
\begin{enumerate}[(i)]
\item \label{i:pexp0} We proceed by induction on $m$. If $m = 0$, then $G(m,T) \cong \Z$ is abelian, so nilpotent.

Let $\bfv = (v_1,\ldots,v_m) \in \ker (T-I)$ be a non-zero element. Since $T-I$ has integer entries, we can pick $\bfv$ in such a way that $v_i/v_j \in \Q$ whenever $v_j \neq 0$. Thus, by rescaling $\bfv$ if necessary, we may assume that $v_i \in \Z$ for all $i$ and $\gcd(v_1,\ldots,v_m) = 1$. Thus $\bfv$ is a direct summand of $\Z^m$, so after a change of basis we may assume that $\bfv = \bfe_m$.

It follows that
\[
T = \begin{pmatrix} \hat{T} & \bfzero \\ \bfu^T & 1 \end{pmatrix}
\]
for some $(m-1) \times (m-1)$ matrix $\hat{T}$ with integer entries and some $\bfu \in \Z^{m-1}$. Note that $\chi_T(x) = (x-1)\chi_{\hat{T}}(x)$, where $\chi_U(x) = \det(xI-U)$ for a matrix $U$, and so all eigenvalues of $\hat{T}$ are equal to $1$. Then it is easy to check that
\[ G(m,T)/\langle a_m \rangle \cong G(m-1,\hat{T}) \]
and so $G(m,T)/\langle a_m \rangle$ is nilpotent by induction hypothesis.

Now we have $[a_m,a_i] = 1$ for all $i$ and $ta_m t^{-1} = \bfa^{T\bfe_m} = \bfa^{\bfe_m} = a_m$, so the element $a_m$ is central in $G(m,T)$. Since $G(m,T)/\langle a_m \rangle$ is nilpotent, it follows that $G(m,T)$ is nilpotent as well.

\item A theorem by Kronecker \cite{kronecker}, whose proof we sketch here, shows that if a monic polynomial $p$ has integer coefficients and all roots on the unit circle, then all roots of $p$ are roots of unity. Indeed, if $p(X) = \prod_{i=1}^m (X-\lambda_i)$ is such a polynomial, then the coefficients of $p$ are symmetric polynomials in the $\lambda_i$ that generate (over $\Q$) the subalgebra of $\Q[\lambda_1,\ldots,\lambda_m]$ consisting of all symmetric polynomials. It follows that for any $n \in \N$, the polynomial $p_n(X) = \prod_{i=1}^m (X-\lambda_i^n)$ has rational coefficients; but as the $\lambda_i$ are algebraic integers, so are the coefficients of $p_n$, and so the coefficients of $p_n$ are integers. Since $|\lambda_i| = 1$ for all $i$, the coefficient of $X^k$ in $p_n$ is bounded by $\genfrac(){0pt}{1}{m}{k}$ for each $k \in \{ 0,\ldots,m \}$. It follows that the set $\{ p_n \mid n \in \N \}$ contains only finitely many polynomials. Thus there exists a polynomial $\hat{p}(X) = \prod_{i=1}^m (X-\nu_i)$ and an infinite subset $I \subseteq \N$ such that $p_n(X) = \hat{p}(X)$ for all $n \in I$. This means that for all $n \in I$, there exists a permutation $\sigma = \sigma_n \in \Sym \{ 1,\ldots,m \}$ such that $\lambda_i^n = \nu_{\sigma(i)}$ for each $i$. Since $\Sym \{ 1,\ldots,m \}$ is finite, there exist two distinct elements $n_1,n_2 \in I$ such that $\sigma_{n_1} = \sigma_{n_2}$. This implies that $\lambda_i^{n_1} = \lambda_i^{n_2}$ for each $i$, and so $\lambda_i$ is $|n_1-n_2|$-th root of unity, as required.

Since the polynomial $\chi_T$ has integer coefficients and all roots on the unit circle, the argument above shows that there exists $n \in \N$ such that all eigenvalues of $T^n$ are equal to $1$. Define a map $\tau_n: G(m,T) \to \Z/n\Z$ by setting $\tau_n(a_i) = 0$ for all $i$ and $\tau_n(t) = 1$: that is, $\tau_n$ is $\tau$ followed by reduction modulo $n$, where $\tau$ is as in \cref{d:GmT}. Now $\ker \tau_n \cong G(m,T^n)$ is nilpotent by part \eqref{i:pexp0} and has index $n$ in $G(m,T)$, and so $G(m,T)$ is virtually nilpotent, as required.

\item It is enough to show that $A(m,T)$ has exponential growth in $G(m,T)$. This follows easily from \cref{l:geps}. Indeed, given an integer $k$ the set
\[
\{ g_{\bfeps}t^{-kR} \mid \bfeps \in \{0,1\}^{k+1} \}
\]
contains $2^{k+1}$ distinct elements of $A(m,T)$, and each of these elements has word length at most $k+1+2kR \leq (k+1)(2R+1)$ over $X = \{ a_1,\ldots,a_m,t \}$. Thus
\[
|A(m,T) \cap B_X((2R+1)(k+1))| \geq 2^{k+1}
\]
for all $k \in \N$, which implies
\[
\liminf_{n \to \infty} \frac{\log |A(m,T) \cap B_X(n)|}{n} \geq \frac{\log 2}{2R+1} > 0,
\]
as required. \qedhere
\end{enumerate}
\end{proof}

\section{Negligibility of the elliptic subgroup} \label{s:A}

In this section we prove \cref{t:A}. As a consequence of \cref{p:Ggrowth}, we restrict to matrices $T$ that have eigenvalues of absolute value not equal to $1$.

Let $G = G(m,T)$ and $A = A(m,T)$. By \cref{p:Ggrowth}, there exists a constant $\alpha > 1$ such that $|A \cap B_Y(n)| \geq \alpha^n$ for all sufficiently large $n$. Fix $\beta \in (1,\alpha^{1/2m})$, and let $d = \det T$. Let $Y$ be a finite generating set for $G$.

Let $\varphi: A \to \Q^m$ be as in \cref{l:phi}. Define the sets
\[ \mathcal{Z}_+ = \{ g \in A \mid \beta^{|g|_Y} \leq \|\varphi(g)\|_\infty \} \]
and
\[ \mathcal{Z}_- = \begin{cases} \{ g \in A \mid d^{\lfloor |g|_Y \log\beta/\log |d| \rfloor}\varphi(g) \notin \Z^m \} & \text{if } |d| \geq 2, \\ \varnothing & \text{if } |d| = 1, \end{cases} \]
and let $\mathcal{Z} = \mathcal{Z}_+ \cup \mathcal{Z}_-$.

\begin{lem} \label{l:Zgen}
Elements of $\mathcal{Z}$ are generic in $A$ with respect to $Y$: that is, $\frac{|\mathcal{Z} \cap B_Y(n)|}{|A \cap B_Y(n)|} \to 1$ as $n \to \infty$.
\end{lem}

\begin{proof}
Suppose first that $|d| \neq 1$. Note that for each $k_1 > 0$ and $k_2 \in \N$, the number of elements in the set
\[ \{ \bfu \in \Q^m \mid \|\bfu\|_\infty \leq k_1, k_2\bfu \in \Z^m \} \subseteq \Q^m \]
is at most $(2k_1k_2+1)^m \leq 3^mk_1^mk_2^m$. As $|d|^{\log\beta/\log |d|} = \beta$ and $\beta^{2m} < \alpha$, we get
\begin{align*}
\limsup_{n \to \infty} \frac{|(A \setminus \mathcal{Z}) \cap B_Y(n)|}{|A \cap B_Y(n)|} &\leq \limsup_{n \to \infty} \frac{|\{ \bfu \in \Q^m \mid \beta^n > \|\bfu\|_\infty, d^{\lfloor n \log\beta/\log |d| \rfloor}\bfu \in \Z^m \}|}{|A \cap B_Y(n)|} \\ &\leq 3^m \limsup_{n \to \infty} \left( \frac{\beta^{2m}}{\alpha} \right)^n = 0.
\end{align*}

If instead $|d| = 1$, then $T \in GL_n(\Z)$ and so $A = \{ \bfa^\bfu \mid \bfu \in \Z^n \}$, hence we obtain
\[ \limsup_{n \to \infty} \frac{|(A \setminus \mathcal{Z}) \cap B_Y(n)|}{|A \cap B_Y(n)|} \leq \limsup_{n \to \infty} \frac{|\{ \bfu \in \Z^m \mid \beta^n > \|\bfu\|_\infty \}|}{|A \cap B_Y(n)|} \leq 3^m \limsup_{n \to \infty} \left( \frac{\beta^m}{\alpha} \right)^n = 0, \]
as required.
\end{proof}

For a path $\omega = y_1 \cdots y_n$ (with $y_i \in Y$) in the Cayley graph $\Gamma(G,Y)$, we may consider its image in the quotient $G/A \cong \Z$. In particular, define the \emph{maximal height}, \emph{minimal height} and \emph{total height} of $\omega$ as the numbers
\[ \mathfrak{h}_+(\omega) = \max \{ \tau(y_1 \cdots y_i) \mid 0 \leq i \leq n \}, \]
\[ \mathfrak{h}_-(\omega) = \min \{ \tau(y_1 \cdots y_i) \mid 0 \leq i \leq n \} \]
and
\[ \mathfrak{h}(\omega) = \max\{ \mathfrak{h}_+(\omega),-\mathfrak{h}_-(\omega) \}, \]
respectively, where $\tau$ is as in \cref{d:GmT}. Note that $\mathfrak{h}_-(\omega) \leq 0 \leq \mathfrak{h}_+(\omega)$ for any word $\omega$.

\begin{lem} \label{l:highwords}
There exists a constant $\delta > 0$ such that if $\omega$ is a geodesic word representing $g \in \mathcal{Z}$ with $|g|_Y$ sufficiently large, then $\mathfrak{h}(\omega) \geq \delta |g|_Y$.
\end{lem}

\begin{proof}
Since each element of $Y$ can be expressed as $t^{-r} \bfa^\bfu t^s$ for some $\bfu \in \Z^m$ and $0 \leq r,s \leq c$, it is easy to see -- by induction on the length of $\omega$, say -- that $\varphi(g) \in d^{\mathfrak{h}_-(\omega)-c}\Z^m$. Thus if $g \in \mathcal{Z}_-$ (and so $\mathcal{Z}_- \neq \varnothing$, implying that $|d| \geq 2$) then we have
\[ \mathfrak{h}(\omega) \geq -\mathfrak{h}_-(\omega) > \frac{\log\beta}{\log |d|}|g|_Y - c \geq \frac{\log\beta}{2\log |d|}|g|_Y \]
if $|g|_Y \geq 2c\log|d|/\log\beta$.

Suppose now $g \in \mathcal{Z}_+$. Then we have
\begin{equation} \label{e:gY}
|g|_Y \leq \log \|\varphi(g)\|_\infty / \log \beta.
\end{equation}
Let $\omega = y_1 \cdots y_n$ with $y_i \in Y$, and let $L = \max \{ \|T\|_{op}, \|T^{-1}\|_{op} \}$, where $\| \cdot \|_{op}$ denotes the operator norm with respect to the $\ell_\infty$-norm on $\C^m$. Note that, since $G$ has exponential growth, $T$ has an eigenvalue $\lambda$ with $|\lambda| \neq 1$ by \cref{p:Ggrowth}, and so $L > 1$.

It is easy to show -- by induction on $i$, say -- that
\[ \| \varphi(y_1 \cdots y_i t^{-\tau(y_1 \cdots y_i)}) \|_\infty \leq L^{\mathfrak{h}(y_1 \cdots y_i)} \sum_{j=1}^i \| \varphi(y_jt^{-\tau(y_j)})\|_\infty \]
for $1 \leq i \leq n$, and hence
\begin{equation} \label{e:phig}
\| \varphi(g) \|_\infty \leq c_0 L^{\mathfrak{h}(\omega)} |g|_Y,
\end{equation}
where $c_0 = \max \{ \|\varphi(yt^{-\tau(y)})\|_\infty \mid y \in Y \}$. Combining \eqref{e:gY} and \eqref{e:phig} yields
\[ |g|_Y \leq \frac{1}{\log \beta} \left( \log c_0 + \mathfrak{h}(\omega)\log L + \log |g|_Y \right) \]
and hence
\[ \mathfrak{h}(\omega) \geq \frac{1}{\log L} \left( |g|_Y \log\beta - \log |g|_Y - \log c_0 \right). \]
Now if $|g|_Y$ is big enough then we have $\log |g|_Y + \log c_0 \leq \frac{\log \beta}{2} |g|_Y$; substituting this yields
\[ \mathfrak{h}(\omega) \geq \frac{\log\beta}{2\log L}|g|_Y, \]
so by setting $\delta = \log\beta / 2\log(\max\{L,|d|\})$ we are done.
\end{proof}

The next Lemma shows that there exists a particular subset $\mathcal{A} \subseteq \mathcal{Z}$ such that, given an element $g \in \mathcal{A}$, there are `many' words over $Y$ representing $g$ that are `not too long'. More specifically, for $n,p,q \in \N$ with $p \leq n$ and for $h \in \Z$, define
\[
\mathcal{A}(n) = \mathcal{A}_{p,h}(n) = \{ g_1 g_2 \mid g_1 \in t^hA \cap B_Y(p), g_2 \in t^{-h}A \cap B_Y(n-p) \} \cap \mathcal{Z},
\]
and define the map
\begin{align*}
\mu_{n,p,q,h}: [t^hA \cap B_Y(p+q)] \times [t^{-h}A \cap B_Y(n-p+q)] &\to A \cap B_Y(n+2q), \\
(g_1,g_2) &\mapsto g_1g_2.
\end{align*}
Furthermore, let $X = \{ a_1,\ldots,a_m,t \}$ be the standard generating set for $G$. Since the word metrics $|\cdot|_X$ and $|\cdot|_Y$ are bi-Lipschitz equivalent, there exists a constant $c \in \N$ such that
\[
|g|_X \leq c |g|_Y \qquad \text{and} \qquad |g|_Y \leq c|g|_X
\]
for all $g \in G$.

\begin{lem} \label{l:manywords}
Let $R \in \N$ be given by \cref{l:geps}. Then for any $\hat\delta > 0$, any $n,p,k \in \N$ with $p \leq n$ and $k \leq \hat\delta n / R$, any $h \in \Z$ with $|h| \geq \lfloor \hat\delta n \rfloor$, and any $g \in \mathcal{A}_{p,h}(n)$, we have
\[
|(\mu_{n,p,c(c+1)k,h})^{-1}(g)| \geq \genfrac(){0pt}{1}{\lfloor \hat\delta n / R \rfloor}{k}.
\]
\end{lem}

\begin{proof}
Let $\mathcal{A} = \mathcal{A}_{p,h}(n)$ and $\mu = \mu_{n,p,c(c+1)k,h}$. Let $\mathcal{P}$ be the set of all subsets of $\{1,\ldots,\lfloor \hat\delta n / R \rfloor\}$ of cardinality $k$. We will find an injection $\mathcal{P} \to \mu^{-1}(g), B \mapsto (l_{B,1},l_{B,2})$, which will prove our claim.

Suppose first that $h > 0$. As $g \in \mathcal{A}$, there exists an expression $g = g_1g_2$, where $g_1 \in t^hA \cap B_Y(p)$ and $g_2 \in t^{-h}A \cap B_Y(n-p)$. Let $\omega_1 = y_1 \cdots y_r$ and $\omega_2 = y_{r+1} \cdots y_{|g|_Y}$ be geodesic paths representing $g_1$ and $g_2$, respectively, where $y_i \in Y$. Given $B \in \mathcal{P}$, we `modify' $(g_1,g_2)$ in a way that preserves the product $g_1g_2$, as follows; this construction is illustrated in \cref{f:k1k2}.

\begin{figure}[ht]

\begin{tikzpicture}[scale=0.65]
\filldraw [fill=green!10] (0.8,2) rectangle (10,3);
\draw [black,thick] plot [smooth] coordinates { (0,0) (1,1) (2,4.5) (3,4) (4,6) (5,7) (6,7.5) (7,7) (8,5) (9,4) (9.5,3) (10,0) (11,-2) (11.5,0) };
\filldraw (7,7) circle (2pt);
\filldraw (0,0) circle (2pt);
\filldraw (11.5,0) circle (2pt);
\draw [thick,->] (1.998,4.497) -- (2,4.5) node [left] {$\omega_1$};
\draw [thick,->] (8.997,4.006) -- (9,4) node [right] {$\omega_2$};
\draw [very thin] (0,0) -- (11.5,0);
\draw [<->] (7,0) -- (7,6.95) node [pos=0.6,right] {$h$};
\draw [<->] (3.8,0) -- (3.8,2) node [pos=0.6,left] {$b_iR-\frac{c}{2}\!$};
\draw [<->] (4,0) -- (4,3) node [pos=0.4,right] {$\!b_iR+\frac{c}{2}$};

\begin{scope}[xshift=13cm,yshift=4.25cm,yscale=1.25]
\fill [green!10] (-0.5,0) rectangle (10.5,3);

\draw [dashed] (-1,1.5) -- (4.5,1.5);
\draw [dash pattern=on 1pt off 3pt] (4.5,1.5) -- (6.5,1.5);
\draw [dashed] (6.5,1.5) -- (10,1.5);
\draw [|->|] (1.9,1.5) -- (1.9,2.1) node [midway,right] {$q_{i,1}$};
\draw [|->|] (8.8,1.5) -- (8.8,0.6) node [midway,right] {$q_{i,2}$};
\draw [->] (-1,-0.5) -- (-1,3.5) node [above left] {$\tau$} node [midway,left] {$b_iR$};

\draw (4.5,3) -- (-0.5,3) -- (-0.5,0) -- (4.5,0);
\draw (6.5,3) -- (10.5,3) -- (10.5,0) -- (6.5,0);
\draw [dash pattern=on 1pt off 2pt] (4.5,3) -- (6.5,3);
\draw [dash pattern=on 1pt off 2pt] (4.5,0) -- (6.5,0);
\draw [thick,dashed] (0.8,-0.6) -- (1,0);
\draw [thick,->] (1,0) -- (1.2,0.6) node [left] {$y_{j_{i,1}}$};
\filldraw [thick] (1.2,0.6) -- (1.7,2.1) circle (2pt) -- (2,3);
\draw [thick,dashed,->] (2,3) -- (2.2,3.6) node [left] {$y_{j_{i,1}+1}$};
\draw [thick,dashed] (2.2,3.6) -- (2.29,3.87);
\draw [thick,dashed] (7.85,3.6) -- (8,3);
\draw [thick,->] (8,3) -- (8.2,2.2) node [right] {$y_{j_{i,2}}$};
\filldraw [thick] (8.2,2.2) -- (8.6,0.6) circle (2pt) -- (8.75,0);
\draw [thick,dashed,->] (8.75,0) -- (8.9,-0.6) node [right] {$y_{j_{i,2}+1}$};
\draw [thick,dashed] (8.9,-0.6) -- (9,-1);

\draw[->,very thick,decorate,decoration=snake] (5.5,-0.25) -- (5.5,-1.75);
\end{scope}

\begin{scope}[xshift=13cm,yshift=-2cm,yscale=1.25]
\fill [green!10] (-0.5,0) rectangle (10.5,3);

\draw [dashed] (3.2,1.5) -- (4.5,1.5);
\draw [dash pattern=on 1pt off 3pt] (4.5,1.5) -- (6.5,1.5);
\draw [dashed] (6.5,1.5) -- (7.1,1.5);
\draw [red,thick,->] (7.1,1.5) -- (7.85,1.5) node [above] {$a_j$};
\draw [red,thick] (7.85,1.5) -- (8.6,1.5);
\draw [blue,thick,->] (7.1,1.5) -- (7.1,1.05) node [right] {$t^{q_{i,2}}$};
\draw [blue,thick] (7.1,1.05) -- (7.1,0.6);
\draw [blue,thick,->] (8.6,1.5) -- (8.6,1.05) node [right] {$t^{q_{i,2}}$};
\draw [blue,thick] (8.6,1.05) -- (8.6,0.6);
\draw [red,thick,->] (1.7,1.5) -- (2.45,1.5) node [below] {$a_j$};
\draw [red,thick] (2.45,1.5) -- (3.2,1.5);
\draw [blue,thick,->] (1.7,1.5) -- (1.7,1.8) node [right] {$t^{q_{i,1}}$};
\draw [blue,thick] (1.7,1.8) -- (1.7,2.1);
\draw [blue,thick,->] (3.2,1.5) -- (3.2,1.8) node [right] {$t^{q_{i,1}}$};
\draw [blue,thick] (3.2,1.8) -- (3.2,2.1);

\draw (4.5,3) -- (-0.5,3) -- (-0.5,0) -- (4.5,0);
\draw (6.5,3) -- (10.5,3) -- (10.5,0) -- (6.5,0);
\draw [dash pattern=on 1pt off 2pt] (4.5,3) -- (6.5,3);
\draw [dash pattern=on 1pt off 2pt] (4.5,0) -- (6.5,0);
\draw [thick,dashed] (0.8,-0.6) -- (1,0);
\draw [thick,->] (1,0) -- (1.2,0.6) node [left] {$y_{j_{i,1}}$};
\filldraw [thick] (1.2,0.6) -- (1.7,2.1) circle (2pt);
\filldraw [thick] (3.2,2.1) circle (2pt) -- (3.5,3);
\draw [thick,dashed,->] (3.5,3) -- (3.7,3.6) node [left] {$y_{j_{i,1}+1}$};
\draw [thick,dashed] (3.7,3.6) -- (3.79,3.87);
\draw [thick,dashed] (7.85,3.6) -- (8,3);
\draw [thick,->] (8,3) -- (8.2,2.2) node [right] {$y_{j_{i,2}}$};
\filldraw [thick] (8.2,2.2) -- (8.6,0.6) circle (2pt);
\filldraw [thick] (7.1,0.6) circle (2pt) -- (7.25,0);
\draw [thick,dashed,->] (7.25,0) -- (7.4,-0.6) node [right] {$y_{j_{i,2}+1}$};
\draw [thick,dashed] (7.4,-0.6) -- (7.5,-1);
\end{scope}
\end{tikzpicture}

\caption{Construction of the words $\psi_{B,1}$ and $\psi_{B,2}$ in the proof of \cref{l:manywords}.} \label{f:k1k2}
\end{figure}
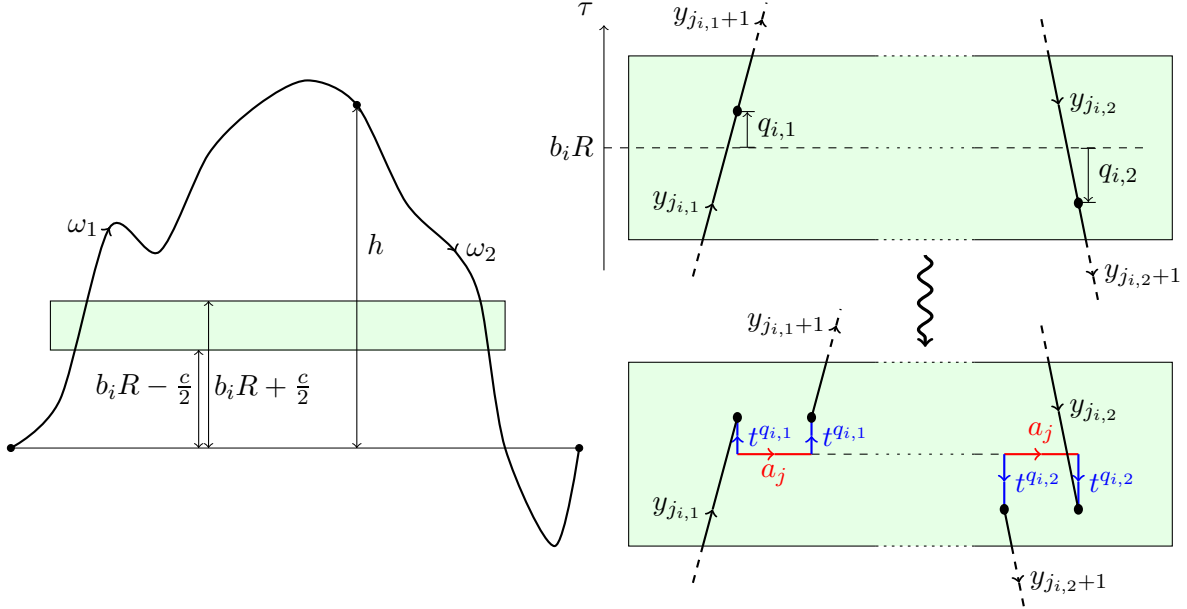

Let $B \in \mathcal{P}$, and write $B = \{ b_1,\ldots,b_k \}$, where $1 \leq b_1 < b_2 < \cdots < b_k \leq \lfloor \hat\delta n / R \rfloor$. Note that $|y|_X \leq c$ for all $y \in Y$, which implies that $|\tau(y)| \leq c$ for all $y \in Y$. Since $\tau(g_1) = -\tau(g_2) = h \geq \lfloor \hat\delta n \rfloor$, it follows that for each $s \in \{1,\ldots,\lfloor \hat\delta n \rfloor\}$, there exist integers $j_1,j_2$ with $0 \leq j_1 \leq r \leq j_2 \leq |g|_Y$ such that $|\tau(y_1 \cdots y_{j_1})-s| \leq \frac{c}{2}$ and $|\tau(y_{j_2+1} \cdots y_{|g|_Y})+s| \leq \frac{c}{2}$. In particular, for each $i \in \{ 1,\ldots,k \}$, there exist integers $j_{i,1},j_{i,2}$ such that
\begin{equation} \label{e:taubR}
|\tau(y_1 \cdots y_{j_{i,1}})-b_iR| \leq \frac{c}{2} \qquad \text{and} \qquad |\tau(y_{j_{i,2}+1} \cdots y_{|g|_Y})+b_iR| \leq \frac{c}{2}.
\end{equation}

Let $j$ be as in \cref{l:geps}, and define the words $\psi_{B,1}$ and $\psi_{B,2}$ as follows. For $\psi_{B,1}$, start with the word $\omega_1$, and for each $i \in \{ 1,\ldots,k \}$, insert a geodesic subword representing $t^{-q_{i,1}} a_j t^{q_{i,1}}$ between $y_{j_{i,1}}$ and $y_{j_{i,1}+1}$, where $q_{i,1} = \tau(y_1 \cdots y_{j_{i,1}})-b_iR$. For $\psi_{B,2}$, start with the word $\omega_2$, and for each $i \in \{ 1,\ldots,k \}$, insert a geodesic subword representing $t^{-q_{i,2}} a_j^{-1} t^{q_{i,2}}$ between $y_{j_{i,2}}$ and $y_{j_{i,2}+1}$, where $q_{i,2} = -\tau(y_{j_{i,2}+1} \cdots y_{|g|_Y})-b_iR$.

Let $l_{B,1}$ and $l_{B,2}$ be the elements represented by words $\psi_{B,1}$ and $\psi_{B,2}$, respectively. Then commutativity of $A$ and the choice of the $q_{i,1}$ and the $q_{i,2}$ implies that $l_{B,1}l_{B,2} = g_1g_2 = g$. Furthermore, it is clear that $\tau(l_{B,1}) = \tau(g_1) = h$ and $\tau(l_{B,2}) = \tau(g_2) = -h$. Finally, by \eqref{e:taubR} we have $|q_{i,1}|,|q_{i,2}| \leq \frac{c}{2}$ for each $i$, so every fragment inserted into $\omega_1$ (or $\omega_2$) to form $\psi_{B,1}$ (or $\psi_{B,2}$) has word-length at most $c+1$ with respect to $X$, and so at most $c(c+1)$ with respect to $Y$. It follows that
\begin{align*}
|l_{B,1}|_Y &\leq |g_1|_Y + c(c+1)k \leq p+c(c+1)k \\ \text{and} \qquad |l_{B,2}|_Y &\leq |g_2|_Y + c(c+1)k \leq n-p+c(c+1)k.
\end{align*}
Thus indeed $(l_{B,1},l_{B,2}) \in \mu^{-1}(g)$, as required.

Finally, to show that the map $\mathcal{P} \to \mu^{-1}(g), B \mapsto (l_{B,1},l_{B,2})$ is injective, we use \cref{l:geps}. Indeed, if $B,C \in \mathcal{P}$ are distinct, then
\[
\varphi(l_{B,1}l_{C,1}^{-1}) = \sum_{i=1}^k (\mathbbm{1}_B(i)-\mathbbm{1}_C(i)) T^{iR} \bfe_j = \varphi(g_{\mathbbm{1}_B}g_{\mathbbm{1}_C}^{-1}),
\]
where $\varphi$ is as in \cref{l:phi} and $\mathbbm{1}_B,\mathbbm{1}_C: \{1,\ldots,k\} \to \{0,1\}$ are the indicator functions for $B$ and $C$, defined by
\[
\mathbbm{1}_B(i) = \begin{cases} 1 & \text{if } i \in B, \\ 0 & \text{otherwise.} \end{cases}
\]
In particular, by \cref{l:phi,l:geps} we have $l_{B,1}l_{C,1}^{-1} \neq 1$, and so $(l_{B,1},l_{B,2}) \neq (l_{C,1},l_{C,2})$, completing the proof in the case $h > 0$.

If instead $h < 0$, then a very similar argument, obtained essentially by `turning \cref{f:k1k2} upside down', works.
\end{proof}

\cref{t:A} can now be deduced from the following Theorem, which at first glance seems to be marginally weaker than \cref{t:A}.

\begin{thm} \label{t:Aalmost}
Let $f: \N \to [0,\infty)$ be a function such that $f(n) \to 0$ as $n \to \infty$. Then
\[
\log |A \cap B_Y(n)| - \log |B_Y(n)| \leq -nf(n)
\]
for all sufficiently large $n$.
\end{thm}

It would seem that \cref{t:A} would only follow from \cref{t:Aalmost} if we were allowed to take $f$ in \cref{t:Aalmost} to be a strictly positive constant function. However, using a general argument on sequences, we can actually deduce \cref{t:A} from \cref{t:Aalmost}.

\begin{proof}[Proof of \cref{t:A}]
Suppose for contradiction that $A(m,T)$ is not exponentially negligible in $G(m,T)$ with respect to a finite generating set $Y$. This is equivalent to saying that
\[ c_i = \frac{\log |A(m,T) \cap B_Y(n_i)|-\log |B_Y(n_i)|}{n_i} \to 0 \qquad \text{as } i \to \infty  \]
for some (without loss of generality, strictly increasing) sequence $(n_i)_{i=1}^\infty$ in $\N$. Note that, as $A(m,T) \neq G(m,T)$, we have $Y \nsubseteq A(m,T)$ and so $c_i < 0$ for all $i$. Define a function $f: \N \to [0,\infty)$ by
\[ f(n) = \begin{cases} -2c_i & \text{if } n = n_i \text{ for some } i, \\ 0 & \text{otherwise.} \end{cases} \]
Then $f(n) \to 0$ as $n \to \infty$, and so it would follow from \cref{t:Aalmost} that $-n_ic_i \geq -2n_ic_i$ for $i$ large enough, which gives a contradiction. Therefore $A(m,T)$ is exponentially negligible in $G(m,T)$ with respect to $Y$, as required.
\end{proof}

\begin{proof}[Proof of \cref{t:Aalmost}]
The proof uses \cref{l:Zgen,l:highwords} to find a subset of $A$ of the form $\mathcal{A}(n)$ (as defined before \cref{l:manywords}) that is `large' in an appropriate sense, and then uses \cref{l:manywords} to give bounds.

We first show that a generic element of $A$ is `not too short' in terms of word length. Let $\mu > 1$ be the growth rate of $G$ with respect to $Y$: that is, $\mu = \limsup_{n \to \infty} \sqrt[n]{|B_Y(n)|}$, and recall that by \cref{p:Ggrowth} there exists $\alpha > 1$ such that $|A \cap B_Y(n)| \geq \alpha^n$ for all sufficiently large $n$. Fix $\zeta \in \left(0,\frac{\log\alpha}{\log(\mu+1)}\right)$. For $n$ sufficiently large we have $|B_Y(\lfloor \zeta n \rfloor)| \leq (\mu+1)^{\zeta n}$ by the definition of $\mu$, thus
\[
\frac{|A \cap B_Y(\lfloor \zeta n \rfloor)|}{|A \cap B_Y(n)|} \leq \frac{|B_Y(\lfloor \zeta n \rfloor)|}{|A \cap B_Y(n)|} \leq \frac{(\mu+1)^{\zeta n}}{\alpha^n} = \left( \frac{(\mu+1)^\zeta}{\alpha} \right)^n
\]
for sufficiently large $n$, and so $\frac{|A \cap B_Y(\lfloor \zeta n \rfloor)|}{|A \cap B_Y(n)|} \to 0$ as $n \to \infty$. It follows from this and \cref{l:Zgen} that
\begin{equation} \label{e:ZB}
\frac{|\mathcal{Z} \cap (B_Y(n) \setminus B_Y(\lfloor \zeta n \rfloor))|}{|A \cap B_Y(n)|} \to 1 \qquad \text{as } n \to \infty.
\end{equation}

We now construct a subset of $|A \cap B_Y(n)|$ of the form $\mathcal{A}(n)$ that contains `enough' elements. For any $g \in G$, choose be a geodesic word $\omega_g$ representing $g$. If $g \in \mathcal{Z}$ and $\zeta n < |g|_Y \leq n$, then it is clear that $\mathfrak{h}(\omega_g) \leq \frac{cn}{2}$; on the other hand, by \cref{l:highwords} we have $\mathfrak{h}(\omega_g) \geq \delta\zeta n$ when $n$ is large enough, for some (universal) constant $\delta > 0$. This gives at most $\frac{cn}{2}$ possible values of $\mathfrak{h}(\omega_g)$, and so by pidgeonhole principle there exists some $h_0 = h_0(n) \in \N$ with $\delta\zeta n \leq h_0 \leq \frac{cn}{2}$ such that
\[
|\{ g \in \mathcal{Z} \cap B_Y(n) \mid \mathfrak{h}(\omega_g) = h_0 \}| \geq \frac{2}{cn} |\mathcal{Z} \cap (B_Y(n) \setminus B_Y(\lfloor \zeta n \rfloor))|.
\]
By the definition of $\mathfrak{h}(\omega_g)$, it follows that for some $h = h(n) \in \{ \pm h_0 \}$, at least a half of the elements in $\{ g \in \mathcal{Z} \cap B_Y(n) \mid \mathfrak{h}(\omega_g) = h_0 \}$ can be written as $g = g_1g_2$, where $g_1 \in t^hA$ and $|g|_Y = |g_1|_Y+|g_2|_Y$. By using the pidgeonhole principle on the set $\{1,\ldots,n-1\}$ of possible values for $p = |g_1|_Y$, we see that
\[
|\mathcal{A}_{p,h}(n)| \geq \frac{1}{cn^2} |\mathcal{Z} \cap (B_Y(n) \setminus B_Y(\lfloor \zeta n \rfloor))|
\]
for some $p = p(n)$, where $\mathcal{A}(n) = \mathcal{A}_{p,h}(n)$ is as defined before \cref{l:manywords}. Combining this with \eqref{e:ZB} yields
\begin{equation} \label{e:Aphn}
|\mathcal{A}(n)| \geq \frac{1}{2cn^2} |A \cap B_Y(n)|
\end{equation}
for all sufficiently large $n$.

By Fekete's Lemma, it follows that if we write
\[ \log |B_Y(n)| = (\log \mu + \varepsilon_n)n, \]
then $\varepsilon_n \geq 0$ for all $n$ and $\varepsilon_n \to 0$ as $n \to \infty$. Note that we also have $\delta\zeta n \leq h_0 \leq cp$, and so $p(n) \geq \delta\zeta n/c$; similarly, $n-p(n) \geq \delta\zeta n/c$. Therefore, $p(n) \to \infty$ and $n-p(n) \to \infty$ as $n \to \infty$, and so there exists a function $F: \N \to [0,\infty)$ such that
\[
F(n) \geq f(n), \frac{\log n}{n}, \varepsilon_{p(n)}, \varepsilon_{n-p(n)} \quad \text{for all } n, \qquad \text{and} \qquad F(n) \to 0 \quad \text{as } n \to \infty.
\]
By replacing $F(n)$ with $\lceil nF(n) \rceil/n$ if necessary, we may furthermore assume that $nF(n) \in \Z$ for all $n$.

We now apply \cref{l:manywords} with $n$ sufficiently large (so that $F(n) \leq \delta\zeta/R$), with $p = p(n)$ and $h = h(n)$ as above, with $\hat\delta = \delta\zeta$ and with $k = nF(n)$. It then follows that $|\mu^{-1}(g)| \geq \genfrac(){0pt}{1}{\lfloor \delta_0 n \rfloor}{k}$ for each $g \in \mathcal{A}(n)$, where $\mu = \mu_{n,p,c(c+1)nF(n),h}$ and $\delta_0 = \delta\zeta/R$, and therefore
\begin{equation} \label{e:Anleq}
\begin{aligned}
|\mathcal{A}(n)| &\leq \frac{|t^hA \cap B_Y(p+c(c+1)nF(n)| \times |t^{-h}A \cap B_Y(n-p+c(c+1)nF(n)|}{\genfrac(){0pt}{1}{\lfloor \delta_0 n \rfloor}{nF(n)}} \\ &\leq \frac{|B_Y(p+c(c+1)nF(n)| \times |B_Y(n-p+c(c+1)nF(n)|}{\genfrac(){0pt}{1}{\lfloor \delta_0 n \rfloor}{nF(n)}} \\ &\leq \frac{|B_Y(p)| \times |B_Y(n-p)| \times |B_Y(c(c+1)nF(n)|^2}{\genfrac(){0pt}{1}{\lfloor \delta_0 n \rfloor}{nF(n)}},
\end{aligned}
\end{equation}
where the last inequality comes from the submultiplicativity of the function $n \mapsto |B_Y(n)|$.

The Theorem can now be deduced from the results above via a few calculations. Considering each term in \eqref{e:Anleq} separately and taking logarithms, we get the following bounds:
\begin{align*}
\log |B_Y(p)| &\leq p\log\mu + p\varepsilon_p, \\
\log |B_Y(n-p)| &\leq (n-p)\log\mu + (n-p)\varepsilon_{n-p}, \\
\log |B_Y(c(c+1)nF(n))| &\leq c(c+1)nF(n)\log|Y|, \\
\text{and} \qquad \log \genfrac(){0pt}{1}{\lfloor \delta_0 n \rfloor}{nF(n)} &= \sum_{i=1}^{nF(n)} \log \left( \frac{\lfloor\delta_0n\rfloor-i+1}{i} \right) \geq nF(n) \log \left( \frac{\lfloor\delta_0n\rfloor-nF(n)}{nF(n)} \right) \\ &\geq nF(n)\log(\delta_0/2-F(n)) - nF(n)\log F(n),
\end{align*}
where the last inequality comes from the bound $\lfloor\delta_0n\rfloor \geq \delta_0n/2$, which is true for $n$ large. Combining these bounds and using \eqref{e:Aphn} and \eqref{e:Anleq} gives
\begin{align*}
-\log 2c &-2\log n + \log |A \cap B_Y(n)| \leq \log |\mathcal{A}(n)| \leq n\log\mu + p\varepsilon_p + (n-p)\varepsilon_{n-p} \\ &+ 2c(c+1)nF(n)\log |Y| - nF(n)\log(\delta_0/2-F(n)) + nF(n)\log F(n),
\end{align*}
which can be rearranged to yield
\begin{align*}
\log |A \cap B_Y(n)| - n\log\mu \leq \log 2c &+2\log n + p\varepsilon_p + (n-p)\varepsilon_{n-p} + 2c(c+1)nF(n)\log |Y| \\ &- nF(n)\log(\delta_0/2-F(n)) + nF(n)\log F(n).
\end{align*}
It follows from the definition of $F(n)$ that, for $n$ sufficiently large, all the terms on the right hand side except for the last one can be bounded by a constant multiple of $nF(n)$. Therefore, we have
\[
\limsup_{n \to \infty} \left( \frac{\log |A \cap B_Y(n)| - n\log\mu}{nF(n)} - \log F(n) \right) < \infty.
\]
But since $\log F(n) \to -\infty$ as $n \to \infty$, and since $\log |B_Y(n)| \geq n \log \mu$ for all $n$, we see that
\[
\frac{\log |A \cap B_Y(n)| - \log |B_Y(n)|}{nF(n)} \leq \frac{\log |A \cap B_Y(n)| - n\log\mu}{nF(n)} \leq -1
\]
for all sufficiently large $n$. As $F(n) \geq f(n)$, this implies the result.
\end{proof}

\section{Degree of nilpotence} \label{s:dn}

In this section we prove \cref{t:dn}. We let $G = G(m,T)$ and $A = A(m,T)$, and let $Y$ be a finite generating set for $G$. We assume, as in the previous section, that $G$ has exponential growth.

To show \cref{t:dn}, we use the following general Lemma.
\begin{lem} \label{l:gennilp}
Suppose, given a group $G$ generated by a finite set $Y$, that there exist a subset $\mathcal{N} \subseteq G$ and a function $f: \Z_{\geq 0} \to [0,1]$ satisfying (for all integers $n \geq 0$)
\begin{equation} \label{e:gennilp1}
\frac{|\mathcal{N} \cap B_Y(n)|}{|B_Y(n)|} \leq f(n)
\end{equation}
and
\begin{equation} \label{e:gennilp2}
\frac{|\{ h \in B_Y(n) \mid [g,h] \in \mathcal{N} \}|}{|B_Y(n)|} \leq f(n)
\end{equation}
for all $g \in G \setminus \mathcal{N}$. Then
\[
\frac{|\{ (x_0,\ldots,x_r) \in B_Y(n)^{r+1} \mid [x_0,\ldots,x_r] \in \mathcal{N} \}|}{|B_Y(n)|^{r+1}} \leq (r+1)f(n)
\]
for all integers $r,n \geq 0$.
\end{lem}

\begin{proof}
Induction on $r$. The base case $r = 0$ follows from \eqref{e:gennilp1}.

For $r \geq 1$, write $\bar{x}$ for $[x_0,\ldots,x_{r-1}]$. Then we have
\begin{align*}
\sum_{\substack{x_0,\ldots,x_{r-1} \in B_Y(n) \\ \bar{x} \in \mathcal{N}}} &\frac{|\{ x_r \in B_Y(n) \mid [\bar{x},x_r] \in \mathcal{N} \}|}{|B_Y(n)|^{r+1}} \\ &\leq |\{ (x_0,\ldots,x_{r-1}) \in B_Y(n)^r \mid \bar{x} \in \mathcal{N} \}| \times \frac{|B_Y(n)|}{|B_Y(n)|^{r+1}} \\ &\leq |B_Y(n)|^r rf(n) \times |B_Y(n)|^{-r} = rf(n)
\end{align*}
by the induction hypothesis, and
\begin{align*}
\sum_{\substack{x_0,\ldots,x_{r-1} \in B_Y(n) \\ \bar{x} \notin \mathcal{N}}} &\frac{|\{ x_r \in B_Y(n) \mid [\bar{x},x_r] \in \mathcal{N} \}|}{|B_Y(n)|^{r+1}} \leq \sum_{\substack{x_0,\ldots,x_{r-1} \in B_Y(n) \\ \bar{x} \notin \mathcal{N}}} \frac{|B_Y(n)|f(n)}{|B_Y(n)|^{r+1}} \\ &\leq |B_Y(n)|^r \times \frac{|B_Y(n)|f(n)}{|B_Y(n)|^{r+1}} = f(n)
\end{align*}
by \eqref{e:gennilp2}. This gives 
\begin{align*}
&\frac{|\{ (x_0,\ldots,x_r) \in B_Y(n)^{r+1} \mid [x_0,\ldots,x_r] \in \mathcal{N} \}|}{|B_Y(n)|^{r+1}} \\ &\qquad = \sum_{x_0,\ldots,x_{r-1} \in B_Y(n)} \frac{|\{ x_r \in B_Y(n) \mid [\bar{x},x_r] \in \mathcal{N} \}|}{|B_Y(n)|^{r+1}} \leq rf(n)+f(n) = (r+1)f(n),
\end{align*}
as required.
\end{proof}

For a subgroup $H \leq G$, let $\{1\} = Z_0(H) \unlhd Z_1(H) \unlhd Z_2(H) \unlhd \cdots$ be the upper central series for $H$. We will use \cref{l:gennilp} with $\mathcal{N} = \bigcup_{i=0}^\infty Z_i(H)$ for a particular subgroup $H$. To show \eqref{e:gennilp1}, the following Lemma will be enough.

\begin{lem} \label{l:ZHinA}
Let $H = \langle a_1,\ldots,a_m,t^N \rangle \leq G$, where $N \in \N$. Then there exists a polynomial $p$ such that
\[
\left|\mathcal{N} \cap B_Y(n)\right| \leq p(n)
\]
for all $n \in \N$, where $\mathcal{N} = \bigcup_{i=0}^\infty Z_i(H)$.
\end{lem}

\begin{proof}
Recall that every element of $G$ can be expressed as $t^{-r}\bfa^\bfu t^s$ for some $r,s \geq 0$ and $\bfu \in \Z^m$. We first show that $\mathcal{N} \subseteq A$. Suppose for contradiction that there exists an element $g = t^{-r}\bfa^\bfu t^s \in Z_i(H)$ with $r \neq s$; after replacing $g$ with $g^{-1}$ if necessary, we may assume that $r > s$. Then, for all $\bfv \in \Z^m$ we have $\bfa^\bfv \in H$, and so
\[
\bfa^\bfv Z_{i-1}(H) = g \bfa^\bfv g^{-1} Z_{i-1}(H) = t^{-r} \bfa^{\bfu + T^s\bfv - \bfu} t^{-r} Z_{i-1}(H) = t^{s-r} \bfa^\bfv t^{r-s} Z_{i-1}(H).
\]
Moreover, since $g \in H$ we have $N \mid \tau(g) = s-r$, and so $t^{s-r} \in H$. This implies that $t^{s-r} \in Z_i(H)$, and so any $(i+1)$-fold simple commtator $[g_1,t^{s-r},g_2,\ldots,g_i]$ vanishes for $g_1,\ldots,g_i \in H$.

In particular, since $[ \bfa^\bfu,t^{s-r}] = \bfa^{(T^{r-s}-I)\bfu}$, we have
\begin{equation} \label{e:comm}
1 = [ \bfa^\bfu, \overbrace{t^{s-r},\ldots,t^{s-r}}^i ] = \bfa^{(T^{r-s}-I)^i\bfu}
\end{equation}
for all $\bfu \in \Z^m$, and so $(T^{r-s}-I)^i = 0$. Thus all eigenvalues of $T^{r-s}$ are equal to $1$, and so all eigenvalues of $T$ are roots of unity. As we assumed that $G$ has exponential growth, this contradicts \cref{p:Ggrowth}. Thus $t^{-r}\bfa^\bfu t^s \notin Z_i(H)$ whenever $r \neq s$, and so $Z_i(H) \leq A$, as required.

Therefore, $\mathcal{N} \subseteq A$. If $\bfa^\bfu \in Z_i(H)$ then \eqref{e:comm} holds (with $r = N$ and $s = 0$, say) and so $\bfu \in U \cap \Z^m$, where $U = \bigcup_{i=1}^\infty \ker(T^N-I)^i \leq \C^m$. Therefore,
\begin{equation} \label{e:whatisz}
Z_i(H) = \{ t^{-r}\bfa^{\bfu}t^r \mid \bfu \in \Z^m, r \geq 0, (T^N-I)^i\bfu = \bfzero \}.
\end{equation}
The strategy of the proof is now to show that if $\| \cdot \|$ is a norm on $U$,  then $\|\bfu\|$ will be bounded by a polynomial in $|\bfa^\bfu|_Y$ when $\bfu \in U \cap \Z^m$.

Consider the Jordan normal form $PT^NP^{-1}$ for $T^N$ (where $P \in GL_m(\C)$): it is a block-diagonal matrix with blocks $\hat{X},\hat{Y}_{\hat{y}_1},\ldots,\hat{Y}_{\hat{y}_{\hat{z}}}$ for some $\hat{y}_1,\ldots,\hat{y}_{\hat{z}} \in \N$ with (without loss of generality) $\hat{y}_1 \leq \cdots \leq \hat{y}_{\hat{z}}$, where $X \in GL_{\hat{x}}(\C)$ has no eigenvalues equal to $1$ and
\begin{equation*}
\hat{Y}_{\hat{y}} = \begin{pmatrix}
1 & 1 \\
& 1 & \ddots \\
&& \ddots & 1 \\
&&& 1
\end{pmatrix} \in GL_{\hat{y}}(\C)
\end{equation*}
for $\hat{y} \in \N$. For an element $\bfu \in \C^m$, write
\[
P\bfu = (u_1,\ldots,u_{\hat{x}},u_{1,1},\ldots,u_{1,\hat{y}_1},\ldots,u_{\hat{z},1},\ldots,u_{\hat{z},\hat{y}_{\hat{z}}}).
\]
Define a seminorm on $\C^m$ by setting
\[
\| \bfu \| = \| \pi_U P\bfu \|_\infty = \max \{ |u_{i,j}| \mid 1 \leq i \leq \hat{z}, 1 \leq j \leq \hat{y}_i \}
\]
where $\pi_U$ is the projection of $\C^m$ onto $PU$ obtained by setting the first $\hat{x}$ coordinates of $PU$ to zero, and note that $\| \cdot \|$ becomes a norm when restricted to $U$. Note that 
\begin{equation*}
\hat{Y}_{\hat{y}}^n = \begin{pmatrix}
1 & n & \genfrac(){0pt}{1}{n}{2} & \cdots & \genfrac(){0pt}{1}{n}{\hat{y}-1} \\
& 1 & n & \cdots & \genfrac(){0pt}{1}{n}{\hat{y}-2} \\
&& \ddots & \ddots & \vdots \\
&&& 1 & n \\
&&&& 1
\end{pmatrix}
\end{equation*}
for all $\hat{y} \in \N$ and $n \in \Z$, where for any $n \in \R$ and $r \in \N$ we define $\genfrac(){0pt}{1}{n}{r} = \prod_{i=1}^r \frac{n-i+1}{i}$. It follows from the above that
\[
\|T^{nN}\bfv\| \leq \|\bfv\| \sum_{i=0}^{\hat{y}-1} \left| \genfrac(){0pt}{1}{n}{i} \right|
\]
for all $\bfv \in \C^m$, where $\hat{y} = \hat{y}_{\hat{z}}$. Since $\left| \genfrac(){0pt}{1}{\hat{n}}{i} \right| \leq \genfrac(){0pt}{1}{\hat{n}+i}{i} \leq \genfrac(){0pt}{1}{n+i}{i}$ whenever $n \geq \hat{n} \geq 0$ and since $\genfrac(){0pt}{1}{\hat{n}}{i} = (-1)^i \genfrac(){0pt}{1}{-\hat{n}+i-1}{i}$ for $\hat{n} < 0$, it follows that
\[
\|T^{\hat{n}N}\bfv\| \leq \|\bfv\| p_0(n),
\]
where $p_0(n) = \sum_{i=0}^{\hat{y}-1} \genfrac(){0pt}{1}{n+i}{i}$, whenever $n \in \R$ and $\hat{n} \in \Z$ with $|\hat{n}| \leq n$. Since any element $g \in B_Y(n)$ has $\mathfrak{h}(\omega_g) \leq cn$ (where $\omega_g$ is a geodesic word representing $g$), it follows that if $\bfa^\bfu \in B_X(n)$ then
\[
\| \bfu \| \leq d_0np_0(cn/N)
\]
where $d_0 = \max \{ \|\bfe_i\| \mid 1 \leq i \leq m \}$. Now since $U \cap \Z^m$ is free abelian, we may pick a basis $\bfu_1,\ldots,\bfu_j$ for $U \cap \Z^m$ and define a norm on $\C\langle U \cap \Z^m \rangle$ by
\[
\left\| \sum_{i=1}^j \alpha_i \bfu_i \right\|' = \max \{ |\alpha_i| \mid 1 \leq i \leq j \}.
\]
As any two norms on a finite dimensional vector space are bi-Lipschitz equivalent, we obtain
\[
\| \bfu \|' \leq b_0\|\bfu\| \leq b_0d_0np_0(cn/N)
\]
for some constant $b_0 > 0$ whenever $\bfu \in U \cap \Z^m$ and $|\bfa^\bfu|_Y \leq n$. But by the construction of $\| \cdot \|'$, there are at most $(2k+1)^j$ elements $\bfu \in U \cap \Z^m$ with $\| \bfu \|' \leq k$, and so
\[
|\mathcal{N} \cap \langle a_1,\ldots,a_m \rangle \cap B_Y(n)| \leq p_1(n)
\]
where $p_1(n) = (2b_0d_0np_0(cn/N)+1)^j$ is a polynomial.

Finally, since $\mathfrak{h}(\omega_g) \leq cn$ for all $g \in B_Y(n)$, it follows that any element $g \in \mathcal{N} \cap B_Y(n)$ can be written as $g = t^{-cn}\bfa^\bfu t^{cn}$ for some $\bfu \in \mathbb{Z}^m$, and so
\[
\bfa^\bfu = t^{cn} g t^{-cn} \in \mathcal{N} \cap \langle a_1,\ldots,a_m \rangle \cap B_Y(n+2cn).
\]
Since conjugation by $t^{cn}$ gives a bijection, we obtain
\[
|\mathcal{N} \cap B_Y(n)| \leq |\mathcal{N} \cap \langle a_1,\ldots,a_m \rangle \cap B_Y(n+2cn)| \leq p_1(n+2cn),
\]
and the right hand side is a polynomial, as required.
\end{proof}

Now let $N \in \N$ be such that all eigenvalues $\lambda$ of $T$ that are roots of unity satisfy $\lambda^N = 1$. The key fact justifying this choice of $N$ is that if $T^r\bfv = \bfv$ for some $\bfv \in \C^m$ and some $r \neq 0$, then also $T^N\bfv = \bfv$, and so we have an inclusion of eigenspaces $\ker(T^r-I) \leq \ker(T^N-I)$ for every $r \neq 0$. Let $H \leq G$ and $\mathcal{N} \subseteq G$ be as in \cref{l:ZHinA}.

\begin{lem} \label{l:CGig}
For any $g \in G$, let $C_G^{\mathcal{N}}(g) = \{ h \in G \mid [g,h] \in \mathcal{N} \}$. Then there exists a polynomial $q$ such that
\[
|C_G^{\mathcal{N}}(g) \cap B_Y(n)| \leq q(n)
\]
for all $g \in G \setminus A$ and $n \in \N$.
\end{lem}

\begin{proof}
To show this, we will express the condition $[g,h] \in \mathcal{N}$ in terms of a linear equation, and use this and \cref{l:ZHinA} to bound $|C_G^{\mathcal{N}}(g) \cap t^kA \cap B_Y(n)|$ by a polynomial for any $k \in \Z$. By summing over all possible values of $k \in \Z$ the result will then follow easily.

Let $k \in \Z$ and let $g = t^{-r} \bfa^{\bfu} t^s \in G$ with $g \notin A$ (that is, $r \neq s$). For any $h = t^{-r_0} \bfa^{\bfu_0} t^{s_0}$, we have 
\begin{align*}
[g,h] &= t^{-s} \bfa^{-\bfu} t^{r-s_0} \bfa^{-\bfu_0} t^{r_0-r} \bfa^{\bfu} t^{s-r_0} \bfa^{\bfu_0} t^{s_0} \\ &= t^{-s-s_0} (t^{s_0} \bfa^{-\bfu} t^{-s_0}) (t^r \bfa^{-\bfu_0} t^{-r}) (t^{r_0} \bfa^{\bfu} t^{-r_0}) (t^s \bfa^{\bfu_0} t^{-s}) t^{s+s_0} \\ &= t^{-s-s_0}\bfa^{-T^{s_0}\bfu-T^r\bfu_0+T^{r_0}\bfu+T^s\bfu_0} t^{s+s_0}.
\end{align*}
By \eqref{e:whatisz}, it follows that
\begin{equation} \label{e:centcrit}
[g,t^{-r_0} \bfa^{\bfu_0} t^{s_0}] \in Z_i(H) \qquad \Leftrightarrow \qquad (T^N-I)^i \left( (T^{r_0}-T^{s_0})\bfu + (T^s-T^r)\bfu_0 \right) = \bfzero.
\end{equation}

Now take any two elements $h_j = t^{-r_j} \bfa^{\bfu_j} t^{s_j} \in C_G^{\mathcal{N}}(g) \cap t^kA$ for $j = 1,2$ for some $k \in \Z$. Let $i \in \N$ be such that $[h_j,g] \in Z_i(H)$ for $j = 1,2$. Let $\ell = r_1 - r_2$ (and so $\ell = s_1 - s_2$ since $\tau(h_1) = \tau(h_2) = k$). Then by \eqref{e:centcrit} we have
\begin{align*}
(T^N-I)^i(T^s-T^r)\bfu_1 &= (T^N-I)^i(T^{s_1}-T^{r_1})\bfu = (T^N-I)^i(T^{s_2}-T^{r_2})T^\ell\bfu \\ &= T^\ell(T^N-I)^i(T^{s_2}-T^{r_2})\bfu = T^\ell(T^N-I)^i(T^s-T^r)\bfu_2,
\end{align*}
and so $(T^N-I)^i(T^s-T^r)(T^\ell\bfu_2-\bfu_1) = \bfzero$. Since $T$ is invertible over $\Q$, this gives
\[
(T^N-I)^i(T^{s-r}-I)(T^\ell\bfu_2-\bfu_1) = \bfzero.
\]

Now by the choice of $N$, and since by assumption on $g$ we have $r \neq s$, it follows that $\ker(T^{s-r}-I) \leq \ker(T^N-I)$, and so $(T^N-I)^{i+1}(T^\ell\bfu_2-\bfu_1) = \bfzero$. Hence, by \eqref{e:whatisz}, either $\bfa^{\bfu_1-T^\ell\bfu_2} \in Z_{i+1}(H)$ (if $\ell \geq 0$) or $t^{\ell}\bfa^{T^{-\ell}\bfu_1-\bfu_2}t^{-\ell} \in Z_{i+1}(H)$ (if $\ell < 0$), and since $Z(H)$ is characteristic in $H$ (and so normal in $G$), the conjugate $h_1h_2^{-1}$ is also in $Z_{i+1}(H)$. Note that if $h_1,h_2 \in B_Y(n)$ then we have $|h_1h_2^{-1}|_Y \leq 2n$, and so \cref{l:ZHinA} implies (by fixing $h_1$ and letting $h_2$ vary) that
\[
|C_G^{\mathcal{N}}(g) \cap t^kA \cap B_Y(n)| \leq |\mathcal{N} \cap B_Y(2n)| \leq p(2n)
\]
for a polynomial $p$ (which does not depend on $g$ or $k$). Since any $h \in B_Y(n)$ has $|\tau(h)| \leq cn$, at most $2cn+1$ cosets of $A$ intersect $B_Y(n)$, and so we have
\[
|C_G^{\mathcal{N}}(g) \cap B_Y(n)| \leq (2cn+1)p(2n),
\]
where the right hand side is a polynomial, as required.
\end{proof}

\begin{proof}[Proof of \cref{t:dn}]
By \cref{t:A}, there exists $\beta > 1$ such that $\frac{|A \cap B_Y(n)|}{|B_Y(n)|} \leq \beta^{-n}$ for all $n \in \N$. Let $\alpha \in (1,\beta)$, and let $N$, $H$ be as above. We aim to show that $\gamma_n^Y(\mathcal{N}_r(G)) \leq \alpha^{-n}$ for all sufficiently large $n$. In order to prove this, we will show that the conditions of \cref{l:gennilp} are satisfied with $\mathcal{N}$ as above and a function $f$ such that $f(n) \leq \beta^{-n}$ for $n$ sufficiently large. This will imply the result.

For \eqref{e:gennilp1}, the inequality is immediate by \cref{l:ZHinA}. For \eqref{e:gennilp2}, the inequality is immediate by \cref{l:CGig} if $g \in G \setminus A$, hence we are left with the case $g \in A \setminus \mathcal{N}$. Since $A$ is abelian we have
\[
\frac{|\{ h \in B_Y(n) \cap A \mid [g,h] \in \mathcal{N} \}|}{|B_Y(n)|} = \frac{|A \cap B_Y(n)|}{|B_Y(n)|} \leq \beta^{-n}
\]
for all $n \in \N$ and $g \in A$, so it is enough to show that $\{ h \in G \mid [g,h] \in \mathcal{N} \} \subseteq A$ for all $g \in A \setminus \mathcal{N}$.

Thus, let $g \in A \setminus \mathcal{N}$, and suppose that $[g,h] \in \mathcal{N}$ for some $h = t^{-r_0}\bfa^{\bfu_0}t^{s_0} \in G$: $[g,h] \in Z_i(H)$, say. Then \eqref{e:centcrit} implies that $(T^N-I)^i(T^{r_0-s_0}-I)\bfu = \bfzero$, where $\bfu \in \Z^m$ is such that $\bfa^\bfu$ is conjugate to $g$. If we had $r_0 \neq s_0$ then we would have $\ker(T^{r_0-s_0}-I) \leq \ker (T^N-I)$ by the choice of $N$, and so $(T^N-I)^{i+1}\bfu = \bfzero$. This would imply that $g \in Z_{i+1}(H)$, contradicting $g \notin \mathcal{N}$. Hence indeed $r_0 = s_0$ and so $h \in A$, as claimed.
\end{proof}

\bibliographystyle{amsplain}
\bibliography{all}

\end{document}